\documentclass[12pt]{amsart}
\usepackage{amsmath,amssymb,amscd,amsthm}
\usepackage{txfonts,graphicx,url,bm}
\usepackage{psfrag}
\usepackage{cite}
\usepackage{tikz,subfigure}
\usepackage{xy}
\usepackage{hyperref}
\usepackage{array}
\numberwithin{equation}{subsection}

\newcolumntype{C}{>{$}c<{$}}
\newcolumntype{R}{>{$}r<{$}}
\newcolumntype{L}{>{$}l<{$}}

\newtheorem{theorem}{Theorem}[section]
\newtheorem{proposition}[theorem]{Proposition}
\newtheorem{corollary}[theorem]{Corollary}

\newtheorem{lemma}[theorem]{Lemma}

\theoremstyle{remark}
\newtheorem{remark}[theorem]{Remark}
\newtheorem{example}[theorem]{Example}

\theoremstyle{definition}
\newtheorem{definition}[theorem]{Definition}

\def\beq{\begin{eqnarray}}
\def\eeq{\end{eqnarray}}
\def\bes{\begin{eqnarray*}}
\def\ees{\end{eqnarray*}}
\def\varphihat{\bm\varphi}

\def\muhat{{\bm \mu}}

\def\s{{\mathbf{s}}}

\def\calF{\mathcal{F}}
\def\bH{\mathbb{H}}
\def\bfX{\mathbf{X}}
\def\bfY{\mathbf{Y}}

\def\C{\mathbb{C}}

\def\calA{{\mathcal{A}}}

\def\X{{\mathbb{X}}}

\def\calE{{\mathcal{E}}}

\def\calF{{\mathcal{F}}}

\def\x{\mathbf{x}}
\def\y{\mathbf{y}}

\def\v{\mathbf{v}}
\def\u{\mathbf{u}}
\def\e{\mathbf{e}}
\def\w{\mathbf{w}}

\def\N{\mathbb{Z}_{\geq 0}}

\def\F{\mathbb{F}}
\def\Q{\mathbb{Q}}

\def\calC{{\mathcal C}}

\def\Z{\mathbb{Z}}

\def\gl{{\mathfrak g\mathfrak l}}

\newcommand{\nc}{\newcommand}

\def\bP{{\bf P}}

\nc{\op}[1]{\mathop{\mathchoice{\mbox{\rm #1}}{\mbox{\rm #1}}
{\mbox{\rm \scriptsize #1}}{\mbox{\rm \tiny #1}}}\nolimits}
\nc{\al}{s}

\nc{\ep}{\varepsilon} \nc{\ga}{\gamma} \nc{\Ga}{\Gamma}
\nc{\la}{\lambda} \nc{\La}{\Lambda} \nc{\si}{\sigma}
\nc{\Sig}{{\Gamma}} \nc{\Om}{\Omega} \nc{\om}{\omega}

\nc{\SL}{{\rm SL}} \nc{\GL}{{\rm GL}} \nc{\PGL}{{\rm PGL}}
\nc{\G}{{\rm G}}

\nc{\cpt}{{\op{cpt}}} \nc{\Dol}{{\op{Dol}}} \nc{\DR}{{\op{DR}}}
\nc{\B}{{\op{B}}} \nc{\Triv}{\op{Triv}} \nc{\Hod}{{\op{Hod}}}
\nc{\Log}{{\op{Log}}} \nc{\Exp}{{\op{Exp}}} \nc{\Est}{E_{\op{st}}}
\nc{\Hst}{H_{\op{st}}} \nc{\Left}[1]{\hbox{$\left#1\vbox to
  10.5pt{}\right.\nulldelimiterspace=0pt \mathsurround=0pt$}}
\nc{\Right}[1]{\hbox{$\left.\vbox to
  10.5pt{}\right#1\nulldelimiterspace=0pt \mathsurround=0pt$}}
\nc{\LEFT}[1]{\hbox{$\left#1\vbox to
  15.5pt{}\right.\nulldelimiterspace=0pt \mathsurround=0pt$}}
\nc{\RIGHT}[1]{\hbox{$\left.\vbox to
 15.5pt{}\right#1\nulldelimiterspace=0pt \mathsurround=0pt$}}

\nc{\bee}{{\bf E}} \nc{\bphi}{{\bf \Phi}}

\begin{document}

\title[Torus orbits and Kac polynomials]{Torus orbits on homogeneous varieties and Kac polynomials of quivers}
\author[P. E. Gunnells]{Paul E. Gunnells}
\address{Department of Mathematics and Statistics\\
University of Massachusetts\\
Amherst, MA  01003}
\email{gunnells@math.umass.edu}

\author[E. Letellier]{Emmanuel Letellier}
\address{Laboratoire de Math\'ematiques Nicolas Oresme/CNRS UMR 61 39\\
Universit\'e de Caen Basse-Normandie\\
Esplanade de la Paix\\
BP 5186\\
14032 Caen\\
France}
\email{emmanuel.letellier@unicaen.fr}
\author[F. Rodriguez Villegas]{Fernando Rodriguez Villegas}
\address{
Abdus Salam International Centre for Theoretical Physics\\
Strada Costiera 11\\
34151 Trieste, Italy
\&
Department of Mathematics\\
University of Texas at Austin\\
Austin, TX 78712
 }
\email{villegas@ictp.it}

\date{2 September, 2013}


\thanks{PG is supported by the NSF grant DMS-1101640. EL is supported
by the grant ANR-09-JCJC-0102-01.  FRV is supported by the NSF grant
DMS-1101484 and a Research Scholarship from the Clay Mathematical
Institute.}

\maketitle
\pagestyle{myheadings}

\begin{abstract}In this paper we prove that the counting polynomials
of certain torus orbits in products of partial flag varieties
coincides with the Kac polynomials of supernova quivers, which
arise in the study of the moduli spaces of certain
irregular meromorphic connections on trivial bundles over the
projective line. We also prove that these polynomials can be expressed
as a specialization of Tutte polynomials of certain graphs providing a
combinatorial proof of the non-negativity of their coefficients.

\end{abstract}
\tableofcontents

\section{Introduction}

\subsection{Quivers and Kac polynomials} Given a field $\kappa$ and
$k$ parabolic subgroups $P_1,\dots,P_k$ of $\GL_r(\kappa)$, we form
the cartesian product of partial flag varieties
$\calF:=(\GL_r/P_1)\times\cdots\times\,(\GL_r/P_k)$ on which $\GL_r$ acts
diagonally by left multiplication. To each parabolic $P_i$ corresponds
a unique partition $\mu^i$ of $r$ (given by the size of the blocks).
From the $k$-tuple $\muhat=(\mu^1,\dots,\mu^k)$ we define in a natural
way (see for instance \cite{hausel-letellier-villegas}) a star-shaped
quiver $\Gamma$ with $k$ legs whose lengths are the lengths of the
partitions $\mu^1,\dots,\mu^k$ minus $1$. We also define from $\muhat$
a dimension vector $\v$ of $\Gamma$ with coordinate $r$ on the central
vertex and coordinates $n-\mu^i_1,n-\mu^i_1-\mu^i_2,\dots$ on the
nodes of the $i$-th leg.  Denote by $Z_r\subset\GL_r$ the one
dimensional subgroup of central matrices. The set $\calE(\kappa)$ of
$\GL_r$-orbits of $\calF$ whose
stabilizer is, modulo $Z_r$, a unipotent group is in bijection with
the isomorphism classes of absolutely indecomposable representations
of $(\Gamma,\v)$ over the field $\kappa$. Hence the size of
$\calE(\F_q)$ coincides with the evaluation at $q$ of the Kac
polynomial $A_{\Gamma,\v}(t)$ of $(\Gamma,\v)$, see \S
\ref{representations}. Now it is known from Crawley-Boevey and van den
Bergh \cite{crawley-etal} that when the dimension vector $\v$ is
\emph{indivisible} (i.e.~the gcd of the parts of the partitions
$\mu^i$, $i=1,\dots,k,$ is one), the polynomial $A_{\Gamma,\v}(t)$
coincides (up to a known power of $q$) with the Poincar\'e polynomial
of some quiver varieties $\mathfrak{M}_{\xi}(\v)$ attached to
$(\Gamma,\v)$. Let us give a concrete description of this quiver variety. Assume given $k$ distinct points $a_1,\dots,a_k\in\C$
and a \emph{generic} tuple $(\calC_1,\dots,\calC_k)$ of semisimple
adjoint orbits of $\gl_r(\C)$ such that the multiplicities of the
eigenvalues of $\calC_i$ is given by the partition $\mu^i$. By
Crawley-Boevey \cite{crawley-mat} we can identify this quiver variety
$\mathfrak{M}_\xi(\v)$ with the moduli space of meromorphic
connections
$$
\nabla=d-\sum_{i=1}^kA_i\frac{dz}{z-a_i}
$$
on the trivial rank $r$ vector bundle over the Riemann sphere
$\mathbb{P}^1$ with residues $A_i\in\calC_i$ for $i=1,\dots,k$ and
with no further pole at $\infty$, i.e. $A_1+\cdots+A_k=0$.

In conclusion, when ${\rm gcd}\, (\mu^i_j)_{i,j}=1$, the counting over
$\F_q$ of the $\GL_r$-orbits of $\calF$ with unipotent stabilizer
(modulo $Z_r$) gives the Poincar\'e polynomial of the moduli space of
some regular connections (i.e connections with simple poles at
punctures $a_1,\dots,a_k$) on the trivial rank $r$ vector bundle over
$\mathbb{P}^1$. In general (i.e. without assuming ${\rm gcd}\,
(\mu^i_j)_{i,j}=1$), it is conjectured (see \cite[Conjecture
1.3.2]{hausel-letellier-villegas2}) that this counting coincides with
the pure part of the mixed Hodge polynomial of the moduli space of
$\C^r$-local systems on $\mathbb{P}^1\smallsetminus\{a_1,\dots,a_k\}$
with local monodromy in semisimple conjugacy classes $C_1,\dots,C_k$
of $\GL_r(\C)$ with $(C_1,\dots,C_k)$ \emph{generic} semisimple of
type $\muhat$.

\subsection{Torus orbits on homogeneous varieties} There is another
geometric counting problem that also arises in this setup.  Let
$T\subset \GL_{r}$ be the maximal torus of diagonal matrices.  We can
consider the enumeration over $\F_q$ of the $T$-orbits in
$\calF$.  In general this is a
very subtle problem, even for the simplest case of a single maximal
parabolic subgroup of $\GL_{r}$ where we would be counting
torus orbits on Grassmannians. This problem is connected to matroids,
configuration spaces of points in projective spaces, generalizations
of the dilogarithm, hypergeometric functions, and moduli spaces of
genus $0$ pointed curves \cite{kapranov, gm, ggms, gs}.

In this paper we show (Theorem \ref{maintheo1}) that the counting
function $E^T(q)$ of the $T$-orbits of $\calF$ whose stabilizer is
equal to $Z_r$ coincides with $A_{\Gamma,\v}$, the Kac polynomial of a
certain quiver $\Gamma$ for a certain dimension vector $\v$ (see
\S\ref{dandelion} for the definitions and a picture of $\Gamma$). As a
consequence, $E^T(q)$ is a monic polynomial in $q$ with non-negative
integer coefficients whose degree is given by an explicit
formula. Moreover, we obtain necessary and sufficient condition for
$E^T(q)$ to be non-zero (Theorem \ref{maintheo1}).

The quiver $\Gamma$ belongs to a class of quivers known as
\emph{supernova quivers} (the name is due to Boalch).  The
corresponding generic quiver varieties $\mathfrak{M}_\xi(\v)$ have the
following explicit interpretation. Given a tuple
$(\calC_1,\dots,\calC_k)$ of semisimple adjoint orbits of $\gl_r(\C)$
of type $\muhat$ as above, it follows from Boalch \cite[Theorem 9.11
\& Theorem 9.16]{Boalch} that $\mathfrak{M}_\xi(\v)$ is isomorphic to
the moduli space of meromorphic connections on the trivial rank $r$
vector bundle over $\mathbb{P}^1$ with $k$ simple poles at
$a_1,\dots,a_k$ with residues in $\calC_1,\dots,\calC_k$, and with an
extra pole of order $2$ whose coefficient in $dz/z^2$ (in a local
trivialization) is a semisimple regular matrix.  Hence using the main
result of \cite{crawley-etal} (on the connection between Kac
polynomial $A_{\Gamma,\v}(t)$ and Poincar\'e polynomials of
$\mathfrak{M}_\xi(\v)$), Boalch's result and the results of this
paper, we end up with an interpretation of $E^T(q)$ as the Poincar\'e
polynomial of the moduli space of certain irregular meromorphic
connections as above on the trivial rank $r$ bundle over
$\mathbb{P}^1$.

\subsection{Graph polynomials}
The second main result of this paper is a refined analysis of the
coefficients of the polynomials $A_{\Gamma,\v}(q)=E^T(q)$. More
precisely, we express $E^T(q)$ as a sum of the specialization
$x=1,y=q$ of the \emph{Tutte polynomial} of certain associated graphs
(see Theorem \ref{maintheo2} and~\S\ref{rem:tutte}). We deduce that
the coefficients of $A_{\Gamma,\v}(q)$ count spanning trees in these
graphs of a given weight, which accounts for their
nonnegativity. 

Recall that Kac conjectured that the coefficients of Kac polynomials
(for any finite quiver) are non-negative \cite{kac}. This conjecture
was proved in in the case of an indivisible dimension vector by
Crawley-Boevey and van den Bergh \cite{crawley-etal} with further case
proved by Mozgovoy \cite{mozgovoy}; it was proved in full generality
by Hausel-Letellier-Villegas \cite{hausel-letellier-villegas3}. The
proofs all give a cohomological interpretation of the coefficients of
the Kac polynomial. Our proof of the non-negativity for Kac
polynomials of the supernova quivers is completely different relying,
as mentioned, on Tutte's interpretation of the coefficients of his
polynomial in terms of spanning trees.  This proof is purely
combinatorial and opens a new approach in understanding the Kac
polynomials.

In a continuation to this paper we will discuss how, in fact, the
whole Tutte polynomial of the associated graphs is related to counting
$T$-orbits of~$\calF$.

\section{Supernova complete bi-partite quivers}

\subsection{Generalities on quivers}

Let $\Gamma$ be a finite quiver, $I$ its set of vertices and $\Omega$
its set of arrows. We assume that $\Gamma$ has no loops. For
$\gamma\in \Omega$ we denote by $h(\gamma)$ (respectively, $t(\gamma)$) the head
(resp., tail) of $\gamma$. A \emph{dimension vector} $\v$ of $\Gamma$
is a tuple $(v_i)_{i\in I}$ of non-negative integers
indexed by $I$.

\subsubsection{Roots}

We now recall some well known properties of roots in quivers.  For
more information, we refer the reader to \cite{kac}.

For $i\in I$, let $\e_i \in \Z^{I}$ be the tuple with coordinate $i$
equal 1 and all other coordinates $0$.  Let ${\bf C}=(c_{ij})_{i,j}$
be the Cartan matrix of $\Gamma$, namely
$$
c_{ij}=\begin{cases} 2&\text{if }i=j\\ 
  -n_{ij}&\text{ otherwise},
         \end{cases}
$$
where $n_{ij}$ is the number of edges joining vertex $i$ to vertex
$j$.  The Cartan matrix determines a symmetric bilinear form
$(\phantom{a},\phantom{a})$ on $\Z^I$ by
$$
(\e_i,\e_j)=c_{ij}.
$$
For $i\in I$, define the fundamental reflection
$s_i\colon \Z^I\rightarrow\Z^I$ by
$$
s_i(\lambda)=\lambda-(\lambda,\e_i)\,\e_i, \quad \lambda \in \Z^{I}.
$$
The Weyl group $W=W_\Gamma$ of $\Gamma$ is defined as the subgroup of
automorphisms $\Z^I\rightarrow\Z^I$ generated by the fundamental
reflections $\{s_{i}\mid i\in I\}$. A vector $\v\in\Z^I$ is
called a \emph{real root} if $\v=w(\e_i)$ for some $w\in W$ and $i\in
I$. Let $M=M_\Gamma$ be the set of vectors $\u\in\N^I-\{0\}$ with connected
support such that for all $i\in I$, we have
$$
(\e_i,\u)\leq 0.
$$
Then a vector $\v\in\Z^I$ is said to be an \emph{imaginary root} if
$\v=w(\delta)$ or $\v=w(-\delta)$ for some $\delta\in M$ and $w\in W$.
Elements of $M$ are called \emph{fundamental imaginary roots}.
We denote by $\Phi=\Phi(\Gamma)\subset\Z^I$ the set of all roots of
$\Gamma$ (real and imaginary).

A root is said to be \emph{positive} if its coordinates are all
non-negative. One can show that an imaginary root is positive if and
only if it is of the form $w(\delta)$ with $\delta\in M$. In
particular the Weyl group $W$ preserves the set of positive imaginary
roots.

For any vector $\u\in\Z^I$ put
$$
\Delta(\u):=-\frac{1}{2}(\u,\u).
$$
We have the following characterization of the imaginary roots \cite[Proposition 5.2]{kacbook}:
\begin{lemma}\label{lem:kaclemma}
 Assume that $\v\in\Phi$. Then $\v$ is imaginary if and only
if $\Delta(\v)\geq 0$.  \label{Dpos} \hfill \qed
\end{lemma}

\subsubsection{Representations}\label{representations}

Let $\kappa$ be a field. A \emph{representation} $\varphihat$ of
$\Gamma$ over $\kappa$ is a finite-dimensional graded $\kappa$-vector
space $V^{\varphihat} := \bigoplus_{i\in I}V_i^{\varphihat}$ and a
collection $(\varphi_\gamma)_{\gamma\in\Omega}$ of linear maps
$\varphi_\gamma\colon V_{t(\gamma)}^{\varphihat}\rightarrow
V_{h(\gamma)}^{\varphihat}$. The vector $\v = ({\rm dim}\, V_i)_{i\in
I}$ is called the \emph{dimension vector} of $\varphihat$. We denote
by ${\rm Rep}_{\Gamma,\v}(\kappa)$ the $\kappa$-vector space of
representations of $\Gamma$ of dimension vector $\v$ over $\kappa$.

For $\varphihat\in {\rm Rep}_{\Gamma,\v}(\kappa)$ and $\varphihat'\in
{\rm Rep}_{\Gamma,\v'}(\kappa)$, we have the obvious notions of
morphism $\varphihat\rightarrow\varphihat'$ and direct sum
$\varphihat\oplus\varphihat'\in {\rm Rep}_{\Gamma,\v+\v'}(\kappa)$. We
say that a representation of $\Gamma$ over $\kappa$ is
\emph{indecomposable} if it is not isomorphic to a direct sum of two
non-zero representations of $\Gamma$ over $\kappa$. An indecomposable
representation of $\Gamma$ over $\kappa$ that remains indecomposable
over any finite field extension of $\kappa$ is called an
\emph{absolutely indecomposable} representation of $\Gamma$ over
$\kappa$.

Recall \cite{kac} that there exists a polynomial
$A_{\Gamma,\v}(t)\in\Z[t]$ such that for any finite field $\F_q$, the
evaluation $A_{\Gamma,\v}(q)$ counts the number of isomorphism classes
of absolutely indecomposable representations of $\Gamma$ of dimension
$\v$ over $\F_q$.  We call $A_{\Gamma,\v}(t)$ the \emph{Kac
polynomial} of $\Gamma$ with dimension vector $\v$.\footnote{In the
literature this polynomial is sometimes called the \emph{$A$-polynomial}.}

\begin{theorem}\label{kac2}
The polynomial $A_{\Gamma,\v}(t)$ satisfies the following properties \cite{kac}:
\begin{enumerate}
\item The polynomial $A_{\Gamma,\v}(t)$ does not depend
on the orientation of the underlying graph of $\Gamma$.
\item The polynomial $A_{\Gamma,\v}(t)$ is non zero if and
only if $\v\in\Phi(\Gamma)$. Moreover $A_{\Gamma,\v}(t)=1$ if and only
if $\v$ is a real root. 
\item If non-zero, the polynomial $A_{\Gamma,\v}(t)$ is monic of degree $\Delta(\v)+1$. 
\item For all $w\in W$, we have $A_{\Gamma,w(\v)}(t)=A_{\Gamma,\v}(t)$.
\end{enumerate}
\end{theorem}

We have also the following theorem (see Hausel-Letellier-Villegas
\cite{hausel-letellier-villegas3}), which was conjectured by Kac
\cite{kac}:

\begin{theorem} The polynomial $A_{\Gamma,\v}(t)$ has non-negative integer coefficients.
\label{HLV}\end{theorem}

For $\v=(v_i)_{i\in I}$ a dimension vector, put 
\[
G_\v:=\prod_{i\in I}\GL_{v_i}(\kappa),
\]
and identify ${\rm Rep}_{\Gamma,\v}(\kappa)$ with
$\bigoplus_{\gamma\in\Omega}{\rm
Mat}_{v_{h(\gamma)},v_{t(\gamma)}}(\kappa)$. Under this identification
the group $G_\v$ acts on ${\rm Rep}_{\Gamma,\v}(\kappa)$ by
simultaneous conjugation:
$$
g\cdot\varphihat=(g_{v_{h(\gamma)}}\varphi_\gamma g^{-1}_{v_{t(\gamma)}})_{\gamma\in\Omega}.
$$
Then two representations are isomorphic if and only if they are
$G_\v$-conjugate. Put
\[
Z_\v = \{(\lambda\cdot {\rm Id}_{v_i})_{i\in I}\in G_\v \mid \lambda \in \kappa^{\times}\}.
\]
The group $Z_\v$ acts trivially on ${\rm Rep}_{\Gamma,\v}(\kappa)$. We
have the following characterization of absolute indecomposibility in
terms of $G_{\v}$ and $Z_\v$:
\begin{proposition} \cite[\S
1.8]{kac} A representation in ${\rm
Rep}_{\Gamma,\v}(\kappa)$ is absolutely indecomposable if and only if
the quotient of its stabilizer in $G_\v$ by $Z_\v$ is a unipotent group.
\label{kacu}\end{proposition}

\subsection{Complete bipartite supernova quivers}\label{dandelion} We
now introduce the main objects of this paper.  For fixed non-negative
integers $r,k,s_1,\dots,s_k$ consider the quiver $\Gamma$ with
underlying graph as in Figure \ref{fig:supernova}.  The subgraph with
vertices $(1),\dots,(r), (1;0),\dots,(k;0)$ is the complete bipartite
graph of type $(r,k)$, i.e.~there is an edge between any two vertices
of the form $(i)$ and $(j;0)$.  We orient all edges toward the
vertices $(1;0),\dots, (k;0)$, and denote by $I$ the set of vertices
of $\Gamma$ and by $\Omega$ the set of its arrows.  We call paths of
the form $(j;s_{j}), (j;s_{j}-1),\dotsc ,(j;0)$ the \emph{long legs}
of the graph, and the edges of the complete bipartice subgraph the
\emph{short legs}.

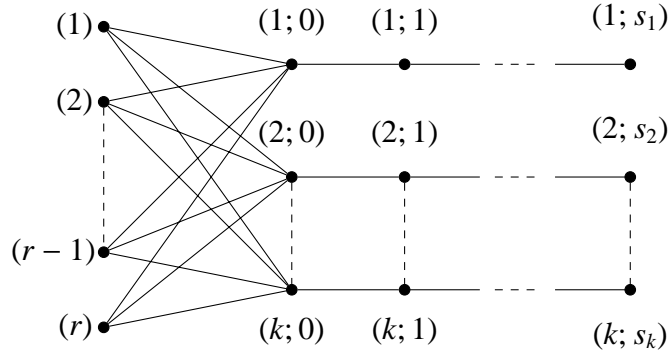
\begin{figure}[htb]
\begin{center}
{
\begin{tikzpicture}
[nodedecorate/.style={shape=circle, fill=black,inner sep=2pt,draw,thick},%
linedecorate/.style={-,thick}]
\tikzstyle{every node}=[draw,circle,fill=black,minimum size=4pt,
                            inner sep=0pt]
 \draw (0,3.5) node (1) [label=left:$(1)$] {}
 (0,2.5) node (2)[label=left:$(2)$]{}
(0, 0.5) node (rm1)[label=left:$(r-1)$]{}
(0,-0.5) node (r)[label=left:$(r)$]{}
(2.5,3) node (1;0)[label=above:$(1;0)$]{}
(2.5,1.5) node(2;0)[label=above:$(2;0)$]{}
(2.5,0) node(k;0)[label=below:$(k;0)$]{}
(4,3) node (1;1)[label=above:$(1;1)$]{}
(4,1.5) node(2;1)[label=above:$(2;1)$]{}
(4,0) node(k;1)[label=below:$(k;1)$]{}
(7,3) node (1;s1)[label=above:$(1;s_1)$]{}
(7,1.5) node(2;s2)[label=above:$(2;s_2)$]{}
(7,0) node(k;sk)[label=below:$(k;s_k)$]{}
;
\draw[dashed] (2)--(rm1) (2;1)--(k;1) (2;s2)--(k;sk) (2;0)--(k;0) (5.2,3)--(5.8,3) (5.2,1.5)--(5.8,1.5) (5.2,0)--(5.8,0);

\draw (1)--(1;0) (1)--(2;0) (1)--(k;0) (2)--(1;0) (2)--(2;0) (2)--(k;0) (rm1)--(1;0) (rm1)--(2;0) (rm1)--(k;0) (r)--(1;0) (r)--(2;0) (r)--(k;0) (1;0)--(1;1) (2;0)--(2;1) (k;0)--(k;1) (1;1)--(5,3) (6,3)--(1;s1) (2;1)--(5,1.5) (6,1.5)--(7,1.5) (4,0)--(5,0) (6,0)--(7,0);
\end{tikzpicture}
}
\end{center}
\caption{The complete bipartite supernova
graph\label{fig:supernova}}
\end{figure}

For $\v\in\N^I$ and $i=1,\dots,k$, define
$$
\delta_i(\v):=-(\e_{(i;0)},\v)=-2v_{(i;0)}+v_{(i;1)}+\sum_{j=1}^rv_{(j)}.
$$

\begin{lemma}\label{lem:twofive}
 Let $\v\in\N^I$. Then $\v$ is in $M_\Gamma$ if and only if the following three conditions are satisfied 

(i) for all $i=1,\dots,k$ we have $\delta_i(\v)\geq 0$,

(ii)  for all $l=1,\dots,r$, 
$$\sum_{j=1}^k v_{(i;0)}\geq 2v_{(l)},
$$

(iii) for all $i=1,\dots,k$ and all $j=0,\dots,s_i-1$, 

\beq
v_{(i;j)}-v_{(i;j+1)}\geq
v_{(i;j+1)}-v_{(i;j+2)}
\label{part}\eeq
with the convention that $v_{(i;s_i+1)}=0$.  \label{fd1}
\end{lemma}

Consider a $k$-tuple of non-zero partitions
$\muhat=(\mu^1,\dots,\mu^k)$, where $\mu^i$ has parts $\mu^i_1\geq
\mu^i_2\geq\cdots\geq\mu^i_{s_i+1}$ with $\mu^i_j$ possibly equal to
$0$. This tuple defines a dimension vector $\v_\muhat=(v_i)_{i\in
I}\in \N^I$ as follows.  Put $v_{(l)}=1$ for $l=1,\dots,r$,
$v_{(i;0)}=|\mu^i|$ and $v_{(i;j)}=|\mu^i|-\sum_{f=1}^j\mu^i_f$ for
$j=1,\dotsc ,s_i$.  Thus the long leg attached to the node
$(i;0)$ (i.e., the type $A_{s_i+1}$ graph with nodes
$(i;0),(i;1),\dots,(i;s_i)$) is labelled with a strictly decreasing
sequence of numbers, and the tips of the short leg are
labelled with $1$.

Notice that for all $i=1,\dots,k$, we have
$$
\delta_i(\v_\muhat)=r-|\mu^i|-\mu^i_1=:\delta(\mu^i),
$$
and that $\v_\mu$ satisfies already the condition Lemma \ref{fd1}
(iii). The condition (ii) is always satisfied unless $k=1$ and
$v_{(1;0)}=1$, in which case $\v_\muhat$ is a real root.  This implies
the following lemma:

\begin{lemma}Assume $k> 1$ or $v_{(1;0)}>1$. Then $\v_\mu\in M_\Gamma$
if and only if for all $i=1,\dots,k$ we have $r\geq |\mu^i|+\mu^i_1$.
\label{fd}\end{lemma}

Recall \cite[Lemma
3.2.1]{hausel-letellier-villegas2} that if
$f=(f_\gamma)_{\gamma\in\Omega}$ is an indecomposable representation
(over an algebraically closed field) of $\Gamma$ of dimension vector
$\v$ and if $v_{(i;0)} >0$ then the linear maps $f_\gamma$, where
$\gamma$ runs over the arrows of the long leg attached to the node
$(i;0)$ are all injective. Recall also (see \S \ref{representations})
that a dimension vector $\v\in\N^I\smallsetminus \{0\}$ is a root of
$\Gamma$ if and only if there exists an indecomposable representation
of $\Gamma$ with dimension vector $\v$. We deduce the following fact:

\begin{lemma} Let $\v\in\N^I$. If $\v\in\Phi(\Gamma)$ and $v_{(i;0)} >0$ then
$v_{(i;0)}\geq v_{(i;1)}\geq v_{(i;2)}\geq\cdots\geq v_{(i;s_i)}$.
\label{monotone}\end{lemma}

\begin{corollary}\label{cor}
Assume that $\v_\muhat$ is an imaginary root. Then $r\geq |\mu^i|$ for all $i=1,\dots,k$.
\end{corollary}

\begin{proof} Since $\v=\v_\mu$ is a positive imaginary root,
$\v'=s_{(i;0)}(\v)$ is also a positive imaginary root. In particular
$v_{(i;0)}'=r-\mu^i_1>0$ and so by Lemma \ref{monotone} we must have
$r-\mu^i_1=v_0'\geq v_{(i;1)}'=v_{(i;1)}=|\mu^i|-\mu^i_1$, i.e.~$r\geq
|\mu^i|$.
\end{proof}

\begin{remark}Corollary \ref{cor} is false for real roots.  For
instance assume $k=1$, $\mu=(3,1)$ and $r=3$. Then clearly
$s_{(1;0)}(\v_\mu)$ is a real root with coordinate $0$ at the vertex
$(1;0)$ and with coordinate $1$ at the edge vertices.  Thus $\v_\mu$
is also a real root, but note that $r<|\mu|$.
\end{remark}

\section{Kac polynomial of complete bipartite supernova quivers}

\subsection{Preliminaries}

\subsubsection{Row echelon forms}\label{row}

Recall that $\kappa$ denotes an arbitrary field. Denote by $B$ the
subgroup of $\GL_n(\kappa)$ of {\bf lower} triangular matrices. Let $r\geq
n$ be an integer. Given a sequence of non-negative integers
$\s=(s_1,s_2,\dots,s_d)$ such that $\sum_i s_i=n$ we denote by $P_\s$
the unique parabolic subgroup of $\GL_n$ containing $B$ and having
$L_\s=\GL_{s_d}\times \dots\times \GL_{s_1}$ as a Levi
factor. Consider a matrix $A\in {\rm Mat}_{n,r}(\kappa)$ and
decompose its set of rows into $d$ blocks; the first block consists
of the first $s_d$ rows of $A$, the second block consists of the
following $s_{d-1}$ rows, and so on.

\begin{definition}\label{def:rowech}
We say that $A$ is in 
\emph{row echelon form} with respect to $\s$ if the following hold:
\begin{enumerate}
\item The rightmost non-zero entry in each row (called a \emph{pivot}) equals
$1$.
\item All entries beneath any pivot vanish.
\item If a block contains two pivots with coordinates $(i,j)$ and
$(i',j')$, then $i<i'$ if and only if
$j<j'$.\label{permcondition}
\end{enumerate}
\end{definition}

We have the following easy proposition, whose proof we leave to the reader:
\begin{proposition}
For any matrix $A\in {\rm Mat}_{n,r}(\kappa)$ of rank $n$ there exists
a unique $g\in P_\s$ such that $gA$ is in  row echelon form with
respect to $\s$.  \label{echelon}\end{proposition}

\subsubsection{Bruhat decomposition}\label{bruhat}

We identify the symmetric group $S_r$ with permutation matrices in
$\GL_r$ (if $w\in S_r$, the corresponding permutation matrix $(a(w)_{ij})_{i,j}$ is defined by $a(w)_{ij}=\delta_{i,w(j)}$). Then $S_r$ acts on the maximal torus $T\simeq (\kappa^\times)^r$ of diagonal matrices as $w\cdot (t_1,\dots,t_r)=(t_{w^{-1}(1)},\dots,t_{w^{-1}(r)})$. Consider a parabolic $P_\s$ of $\GL_r$ for some sequence $\s=(s_1,\dots,s_d)$ with $\sum_is_i=r$ and denote by $S_{r,\s}$
the subset of $S_r$ of permutation matrices which are in row echelon
form with respect to $\s$. Equivalently, if we form  the partition
\begin{equation}
\{1,\dots,r\}=\{1,\dots,s_d\}\cup\{s_d+1,\dots,s_{d-1}\}\cup\cdots\cup\{r-s_1+1,\dots,r\}
\label{part}\end{equation} corresponding to our partition of the rows,
then we have $w^{-1}(i)<w^{-1}(j)$ for any $i<j$ in the same block.
Then we have the following generalized version of the Bruhat decomposition:
$$
\GL_r=\coprod_{w\in S_{r,\s}}P_\s wB.
$$

Denote by
$R$ the root system of $\GL_r$ with respect to $T$.  Recall that it is
the set $\{\alpha_{i,j} \mid 1\leq i\neq j\leq r\}$ of group
homomorphisms $\alpha_{i,j}\colon T\rightarrow\kappa^{\times}$ given
by
$$
\alpha_{i,j}(t_1,\dots,t_r)=t_i/t_j.$$ We have
$\alpha_{i,j}=\alpha_{j,i}^{-1}$ for all $i\neq j$.  The symmetric
group $S_r$ acts on $R$ by $w\cdot \alpha\colon
T\rightarrow\kappa^\times$, $(t_1,\dots,t_r)\mapsto
\alpha(t_{w(1)},\dots,t_{w(r)})$.  In particular
$w\cdot\alpha_{i,j}=\alpha_{w(i),w(j)}$ for all $i\neq
j$. Let  
\[
R^+:=\{ \alpha_{i,j}\mid 1\leq j<i\leq r\}
\]
be the set of positive roots with respect to $B$, and let $R^{-} = R
\smallsetminus R^{+}$.

For $\alpha\in R$, denote by $U_\alpha$ the unique closed one
dimensional unipotent subgroup of $\GL_r$ such that for all $t\in T$
and $g\in U_\alpha$, we have $t(g-1)t^{-1}=\alpha(t)\cdot
(g-1)$. Explicitly, if $\alpha=\alpha_{i,j}$, then the group
$U_\alpha$ consists of matrices of the form $I + xE_{i,j}$, where
$x\in \kappa$ and $E_{i,j}$ is the matrix whose only non-zero entry is
$1$ in position $(i,j)$. We denote by $R_\s\subset R$ the set of roots
$\alpha$ such that $U_\alpha$ is contained in the Levi factor
$L_\s$. For $w\in S_{r,\s}$, put
$$
U_w:=\prod_{\{\alpha \in R^+ \mid w(\alpha )\in R^-\smallsetminus R_\s\}} U_\alpha.
$$
One can show that $U_w$ is a subgroup of $\GL_r$ (see for instance
\cite[10.1.4]{Sp}).  We have the following lemma, whose proof we omit:

\begin{lemma}\label{lem:unicity}
$\phantom{a}$
\begin{enumerate}
\item Any element $g$ in the cell $P_\s wB$ can be written uniquely
as $pwu$ with $p\in P_\s$ and $u\in U_w$.
\item Any element of the form $wu$ with $u\in U_w$ is in its row echelon
form.
\end{enumerate}
\end{lemma}

For any $u\in U_w$ let $u_\alpha$ be its image under the projection
$U_{w}\rightarrow U_\alpha$.  The group $H_\s:=P_\s\times T$ acts on
$wU_w$ via $(p,t)\cdot wu=pwut^{-1}$.  For any group $H$ acting on a
set $X$ and any point $x\in X$, let $C_H(x)\subset H$ denote the
stabilizer of $x$.

\begin{lemma} 
For $u\in U_w$ we have
$$
C_{H_\s}(wu)\simeq C_T(u)=\bigcap_{\{\alpha\in R^+\mid w\cdot\alpha\in
R^-\smallsetminus R_\s,\, u_\alpha\neq 1\}}{\rm Ker}\,\alpha.
$$
\label{ker}
\end{lemma}

\begin{proof} Let $(p,t)\in H_\s$ such that $pwut^{-1}=wu$. Then 
$$
(pwt^{-1}w^{-1})\, w\, (tut^{-1})=wu.
$$
By Lemma \ref{lem:unicity} this identity is equivalent to $p=wtw^{-1}$
and $tut^{-1}=u$ as $T$ normalizes $U_w$. But $tut^{-1}=u$ holds if
and only if for all $\alpha$ we have $tu_\alpha t^{-1}=u_\alpha$ which
identity is equivalent to $t\in {\rm Ker}\,\alpha$ when $u_\alpha\neq
1$.
\end{proof}

\begin{lemma} For any $w\in S_{r,\s}$ we have
\begin{align*}
\{\alpha\in R^+\mid w\cdot \alpha\in R^-\smallsetminus R_\s\}&=\{\alpha\in R^+\mid w\cdot\alpha\in R^-\}\\
&=\{\alpha_{i,j}\mid j<i,\, w(j)>w(i)\}.
\end{align*}
\label{inv}\end{lemma}

\begin{proof} Only the first equality requires proof.  If
$\alpha_{i,j}\in R^+$ and $w\cdot\alpha_{i,j}\in R^-$, i.e., $j<i$ and
$w(i)<w(j)$, then by definition of $S_{r,\s}$ we cannot have
$w(i)$ and $w(j)$ in the same block of the partition
\eqref{part}, i.e., $U_{w\cdot\alpha_{i,j}}=wU_{\alpha_{i,j}}w^{-1}$
is not contained in $L_\s$. We have thus proved that the right hand
side of the first equality is contained in the left hand side.  The
reverse inclusion is easy.
\end{proof}

To simplify notation, we put $U_{i,j}=U_{\alpha_{i,j}}$, so that
$$
U_w=\prod_{j<i,\, w(i)<w(j)}U_{i,j}.
$$

\begin{definition}\label{def:invgraph}
For any $k$-tuple $\w=(w_1,\dots,w_k)\in (S_r)^k$, we denote by $K_\w$
the \emph{inversion graph} of $\w$. Namely, the vertices of $K_\w$ are
labelled by $1,2,\dots,r$ and for any two vertices $i$ and $j$ such
that $j<i$, put an edge from $i$ to $j$ for each $w_{t}$ in $\w $ 
such that $w_t(i)<w_t(j)$.  Thus $K_{\w}$ can have multiple edges.  We
can think of each edge as having one of $k$ possible colors. 

For $\u=(u_1,\dots,u_k)\in U_\w:=U_{w_1}\times\cdots\times U_{w_k}$,
we denote by $K_{\w,\u}$ the subgraph of $K_\w$ that for any pair
$i,j$ includes the edge colored $t$ between vertices $i$ and $j$ if
$(u_t)_{i,j}\neq 1$.
\end{definition}

Denote by $Z_r$ the center of $\GL_r$ and let $T$ act diagonally by
conjugation on $U_\w$.

\begin{proposition} 
For $\u\in U_\w$ we have $C_T(\u)=Z_r$ if and only if the graph
$K_{\w,\u}$ is connected.  \label{connected}
\end{proposition}

\begin{proof}
This is clear  since ${\rm
Ker}\,\alpha_{i,j}$ is the subtorus of elements $(t_1,\dots,t_r)$ such
that $t_i=t_j$.
\end{proof}

\subsection{Computing the Kac polynomials}
\label{Kac-polyns}
Now $\Gamma$ is as in \S \ref{dandelion}. We want to investigate the
polynomial $A_{\Gamma,\v}(t)$. Recall (see Theorem \ref{kac2}) that
$A_{\Gamma,\v}(t)=1$ if $\v$ is a real root and $A_{\Gamma,\v}(t)=0$
if $\v$ is not a root. Moroever $A_{\Gamma,\v}(t)$ is invariant under
the Weyl group action. We are reduced to study the polynomials
$A_{\Gamma,\v}(t)$ with $\v$ is in the fundamental domain $M_\Gamma$. Here we
restrict our study to the case where $\v\in M_\Gamma$ is of the form $\v=\v_\muhat$
for some partition $\muhat$. The important thing for our approach is that
the coordinates of $\v_\muhat$ at the vertices $(j)$, $j=1,\dots,r$, equal
$1$.

Fix once for all a multi-partition $\muhat=(\mu^1,\dots,\mu^k)$ as in
\S \ref{dandelion}, and to alleviate the notation put
$n_i:=|\mu^i|$. We assume that $\v_\muhat$ is in $M_\Gamma$, and so
that $r\geq n_i+\mu^i_1$ for all $i=1,\dots,k$ (see Lemma \ref{fd}).
\bigskip

For a partition $\mu=(\mu_1,\dots,\mu_s)$, we denote by $P_\mu$ the
parabolic subgroup of $\GL_{|\mu|}$ as defined in \S \ref{row} and we denote simply by $S_\mu$
the subset $S_{|\mu|,\mu}$ of the symmetric group $S_{|\mu|}$ as
defined in \S \ref{bruhat}.

\begin{proposition}\label{span}
Assume $\varphihat\in{\rm
Rep}_{\Gamma,\v_\mu}(\kappa)$ is indecomposable. Then
\begin{enumerate}
\item the maps $\varphi_\gamma$, where $\gamma$ runs over the arrows on the $k$ long legs, are all
injective, and
\item for each $i=1,\dots,k$, the images of $\varphi_{(j)\rightarrow (i;0)}$,  with $j=1,\dots,r$,   span $V_{(i;0)}^{\varphihat}$.
\end{enumerate}
\end{proposition}

\begin{proof} 
Let us prove (ii). Let $W_{(i;0)}$ be the subspace generated by the
images of the maps $\varphi_{(j)\rightarrow(i;0)}$ with
$j=1,\dots,r$. If $W_{(i;0)}\subsetneq V_{(i;0)}^\varphi$ we define
subspaces $U_{(i;1)}, U_{(i;2)},\dots, U_{(i;s_i)}$ by
$U_{(i;1)}:=\varphi_{(i;1)\rightarrow(i;0)}^{-1}(W_{(i;0)})$,
$U_{(i;p)}:=\varphi_{(i;p)\rightarrow(i;p-1)}^{-1}(U_{(i;p-1)})$. Let
$\varphihat'$ be the restriction of $\varphihat$ to $$W_{(i;0)}\oplus
\bigoplus_{j=1}^rV_{(j)}^{\varphihat}\oplus\bigoplus_{p=1}^{s_i}U_{(i;p)}\oplus\bigoplus_{f\neq
i}\bigoplus_{j=1}^{s_f}W_{(f;j)}^{\varphihat}.$$Let $W_{(i;0)}'$ be
any subspace such that $V_{(i;0)}^{\varphihat}=W_{(i;0)}\oplus
W_{(i;0)}'$ and define subspaces $U_{(i;j)}'\subset
V_{(i;j)}^{\varphihat}$ by taking the inverse images of
$W_{(i;0)}'$. Then define $\varphihat''$ as the restriction of
$\varphihat$ to $$W_{(i;0)}'\oplus\bigoplus_{p=1}^{s_i}U_{(i;p)}'.$$
Clearly $\varphihat=\varphihat'\oplus\varphihat''$. Hence we must have
$W_{(i;0)}=V_{(i;0)}^{\varphihat}$.
\end{proof}

We denote by $\X_\muhat=\X_\muhat(\kappa)$ the subset of
representations $\varphihat=(\varphi_\gamma)_{\gamma\in\Omega}\in{\rm
Rep}_{\Gamma,\v_\mu}(\kappa)$ that satisfy the conditions (i) and (ii)
in Proposition \ref{span}. As in \S \ref{representations} we identify
${\rm Rep}_{\Gamma,\v_\mu}(\kappa)$ with spaces of matrices and so for
each $i=1,\dots,k$, the coordinates $\varphi_{(1)\rightarrow
(i;0)},\dots,\varphi_{(r)\rightarrow (i;0)}$ of any $\varphihat\in
\X_\muhat$ are identified with non-zero vectors in $\kappa^{n_i}$
which form the columns of a matrix in ${\rm Mat}_{n_i,r}$ of rank
$n_i$. For a partition $\mu=(\mu_1,\dots,\mu_s)$ of $n$, denote by
$G_\mu$ the group $\GL_n\times\GL_{n-\mu_1}\times
\GL_{n-\mu_1-\mu_2}\times\cdots\times\GL_{\mu_s}$. Let $G_\muhat$ be
the subgroup $\prod_{i=1}^kG_{\mu^i}$ of $G_{\v_\mu}$ and denote by
$T$ the $r$-dimensional torus $(\GL_1)^r$. Note that $G_{\v_\mu}\simeq
G_\muhat\times T$.

Denote by $\X_\muhat/G_\muhat$ the set of $G_\muhat$-orbits
of $\X_\muhat$. Since the actions of $G_\muhat$ and $T$ on ${\rm
Rep}_{\Gamma,\v_\muhat}$ commute, we have an action of $T$ on
$\X_\muhat/G_\muhat$.
\bigskip

For $i=1,\dots,k$, put $\mu^i_0:=r-n_i$. Note that
$\tilde{\mu}{^i}:=(\mu_0^i,\mu_1^i,\dots,\mu^i_{s_i})$ is a partition
of $r$, i.e., $\mu_0^i\geq\mu^i_1$. Consider
$$
S_{\tilde{\muhat}}:=S_{\tilde{\mu}^1}\times\cdots\times S_{\tilde{\mu}^k} \subset (S_r)^k,
$$
where $S_\mu$ is defined as in the paragraph preceding Proposition
\ref{span}.

\begin{proposition} 
We have a $T$-equivariant bijection
\begin{equation}\label{eq:bijection}
\X_\muhat/G_\muhat\overset{\sim}{\longrightarrow}\coprod_{\w\in
S_{\tilde{\muhat}}}\w U_\w,
\end{equation}
where $T$ acts on $\w U_\w$ as $t\cdot (w_1u_1,\dots,w_ku_k)=(w_1tu_1t^{-1},\dots,w_ktu_kt^{-1})$. 

\begin{remark} By Lemma \ref{lem:unicity} the right hand side of
  (\ref{eq:bijection}) is isomorphic to $\prod_{i=1}^k
  \GL_r/P_{\tilde{\mu}{^i}}$ on which $T$-acts diagonally by left
  multiplication.

\label{rem:multiflags}\end{remark}

\label{prop1}\end{proposition}

\begin{proof}

We first explain how to construct the bijection
\eqref{eq:bijection}. For each $i=1,\dots,k$, denote by $\calF_{\mu^i}$ the set of partial flags
of $\kappa$-vector spaces
$$
\{0\}\subset E^{s_i}\subset \cdots\subset E^1\subset E^0=\kappa^{n_i}
$$
such that ${\rm dim}\, E^j=n_i-\sum_{f=1}^j\mu^i_f$.  Let $G_{\mu^i}'\subset
G_{\mu^i}$ be the subgroup $\GL_{n_i-\mu^i_1}\times\cdots\times\GL_{\mu^i_{s_i+1}}$ and put $G_\muhat'=\prod_{i=1}^kG_{\mu^i}'$. Let ${\rm Mat}_{n_i,r}'\subset {\rm Mat}_{n_i,r}$ be the subset of
matrices of rank $n_i$. Then we have a natural $\GL_{n_1}\times\cdots\times\GL_{n_k}$-equivariant
bijection
\begin{equation}
\X_\muhat/G_{\muhat}'\simeq \prod_{i=1}^k\left(\calF_{\mu^i}\times {\rm Mat}_{n_i,r}'\right)
\label{bij}
\end{equation} 
that takes a representation
$\varphihat\in\X_\muhat$ to
$(F_{\varphihat}^i,\varphi_{(1)\rightarrow(i;0)},\dots,\varphi_{(r)\rightarrow (i;0)})$; here
$F_{\varphihat}^i$ is the partial flag obtained by taking the images of
the compositions of the $\varphi_\gamma$, where $\gamma$ runs over
the arrows of the $i$-th long leg.

Now fix an element $\varphihat\in\X_\muhat$, and denote by
$(F_{\varphihat},M_{\varphihat})$ its image in 

$$\left(\prod_i\calF_{\mu^i}\right)\times\left(\prod_i {\rm
Mat}_{n_i,r}'\right)$$
via \eqref{bij}. Since we are only interested in the
$G_\muhat$-orbit of $\varphihat$, after taking a $G_\muhat$-conjugate of
$\varphihat$ if necessary we may assume that the stabilizer of
$F_{\varphihat}^i$ is the parabolic subgroup $P_{\mu^i}$ of $\GL_{n_i}$. By
Lemma \ref{echelon} we may further assume that for all $i=1,\dots,k$, the $i$-th coordinate $M_{\varphihat}^i$ of $M_{\varphihat}$ is in
its row echelon form with respect to $(\mu^i_1,\mu^i_2,\dots,\mu^i_{s_i+1})$, this time taking a conjugate $p\cdot
M_{\varphihat}^i$ with $p\in P_{\mu^i}$ if necessary.  It is easy to see
that there is a unique way to complete the matrix $M_{\varphihat}^i$ to
a matrix $\tilde{M}{_{\varphihat}^i}\in\GL_r$ that is in row echelon form with
respect to $(\mu^i_0,\mu^i_1,\dots,\mu^i_{s_i+1})$.
(cf.~Example~\ref{ex:completion}).

Now the pivots of $\tilde{M}{_{\varphihat}^i}$ form a permutation matrix
$w_{\varphihat}^i\in S_{\tilde{\mu}{^i}}$ and $\tilde{M}{_{\varphihat}^i}\in
w_{\varphihat}^iU_{w_{\varphihat}^i}$. We thus defined a map
$X_\muhat/G_\muhat\rightarrow \prod_{i=1}^k\left(\coprod_{w\in S_{\tilde{\mu}{^i}}}wU_w\right)$. The inverse map
is obtained by truncating the last $\mu_0^i$ rows in each coordinate. The fact that the inverse map is $T$-equivariant is easy to
see from the relation $wtut^{-1}=(wtw^{-1})\cdot wu\cdot t^{-1}$.
\end{proof}

\begin{example}\label{ex:completion}
For example, suppose $\s=(1,1)$ and 
$$A=\left(\begin{array}{ccccc}
*&*&1&0&0\\
*&1&0&0&0\end{array}\right).$$
Then
$$
\tilde{A}=\left(\begin{array}{ccccc}
*&*&1&0&0\\
*&1&0&0&0\\
1&0&0&0&0\\
0&0&0&1&0\\
0&0&0&0&1\end{array}\right)$$ is the completion of $A$ to the
corresponding echelon form with respect to $(3,1,1)$.
\end{example}

\begin{proposition} Let $\varphihat\in\X_\muhat$ and let $\w\in S_{\tilde{\muhat}}$, $\u\in U_\w$ such that the image of $\varphihat$ under (\ref{eq:bijection}) is $\w\u$. Then the following assertions are equivalent.

(i) $\varphihat$ is absolutely indecomposable,

(ii) $C_{G_{\v_\muhat}}(\varphihat)=Z_{\v_\muhat}$,

(iii) the graph $K_{\w,\u}$ is connected.
\label{ab}\end{proposition}

By Proposition \ref{span}, the absolutely indecomposable
representations of $(\Gamma,\v_\mu)$ over $\kappa$ are all in
$\X_\mu$.

\begin{proof}[Proof of Proposition \ref{ab}] First assume $\varphihat$
is absolutely indecomposable.  Then $C_{G_{\v_\muhat}}(\varphihat)/Z_{\v_\muhat}$
 is unipotent, see Proposition \ref{kacu}.  Therefore $C_T(\u)$ must reduce to $Z_r$.  Indeed if $t\in C_T(\u)$, then there exists $g\in G_\muhat$ such that $(g,t)\in C_{G_{\v_\muhat}}(\varphihat)$ and so we must have $t\in Z_r$ for $(g,t)$ to be unipotent modulo $Z_{\v_\muhat}$. By Proposition \ref{connected},
the graph $K_{\w,\u}$ is connected.

Now assume that the graph $K_{\w,\u}$ is connected.  By Proposition \ref{kacu} the
representation $\varphihat$ is absolutely indecomposable if and only
if the group $C_{G_\muhat\times T}(\varphihat)/Z_{\v_\muhat}$ is unipotent. Taking a
conjugate of $\varphihat$ if necessary we may assume that
 the image  $(F_{\varphihat},M_{\varphihat})$ under (\ref{bij})  is such  that the stabilizer of
$$
F_{\varphihat}^i=(E^{s_i}_i\subset\cdots\subset E^1_i\subset E^0_i=\kappa^{n_i})
$$
in $\GL_{n_i}$ is the parabolic subgroup $P_{\mu^i}$ and
$M_{\varphihat}^i$ is in its row echelon form with respect to
$(\mu^i_{s_i+1},\mu^i_{s_i},\dots,\mu^i_1)$. Let $(g,t)\in
G_\muhat\times T$ be such that
\begin{equation}(g,t)\cdot\varphihat=\varphihat.
\label{sta}\end{equation} 
Then $g=(g^{(i;t)})_{i,t}\in
G_\muhat$ must
satisfy $g^{(i;0)}\in P_{\mu^i}$ and $g^{(i;t)}=g^{(i;0)}|_{E^t_i}$ for all $i=1,\dots,k$ and $t=1,\dots,s_i$. Taking the image of
$(g,t)\cdot\varphihat=\varphihat$ by (\ref{eq:bijection}) we find that $t\cdot
(\w\u)=\w\u$.  Therefore $t\in C_T(\u)$.

Since (by assumption) $K_{\w,\u}$ is connected, Proposition
\ref{connected} implies $C_T(\u)=Z_r$. Thus (\ref{sta}) reduces to
$$
(\lambda^{-1}\cdot  g^{(i;0)})\cdot M_{\varphihat}^i=M_{\varphihat}^i
$$
for all $i=1,\dots,k$, with $t=\lambda\cdot {\rm I}_r\in Z_r$ for some $\lambda\in
\kappa$. By Proposition \ref{echelon}, we find that $g^{(i;0)}=\lambda\cdot
{\rm I}_n$, i.e., $(g,t)\in Z_{\v_\muhat}$. Hence $C_{G_\mu\times
T}(\varphihat)=Z_{\v_\muhat}$ and therefore $\varphihat$ is absolutely
indecomposable.  This completes the proof.
\end{proof}

For $\w\in S_{\tilde{\muhat}}$ we put 
\begin{equation}\label{eq:rw}
R_\w(q):=\sum_{K\subset K_\w}(q-1)^{b_1(K)},
\end{equation}
where the sum is over the connected subgraphs of $K_\w$; here
$b_1(K)=e(K)-r+1$ is the first Betti number and $e(K)$ is the number
of edges of $K$. If the graph $K_\w$ is not connected then we put
$R_\w(q)=0$.

Denote by $\X_\muhat^\w \subset \X_{\muhat}$ the subset of representations
corresponding to 
$\w U_\w$ in the bijection \eqref{eq:bijection}.

\begin{theorem} The polynomial $R_\w(q)$ counts the number of
isomorphism classes of absolutely indecomposable representations in
$\X_\muhat^\w(\F_q)$.
\label{maintheo2}\end{theorem}

\begin{proof} The $T$-equivariant bijection \eqref{eq:bijection}
induces an isomorphism between the isomorphism classes of $\X_\muhat^\w$
with the $T$-orbits of $\w U_\w$. By Proposition \ref{ab} the isomorphism
classes of absolutely indecomposable representations in $\X_\muhat^\w$
corresponds to the $T$-orbits of $\calC=\{\w\u\in \w U_\w\,|\, K_{\w,\u}
\text{ is connected}\,\}$. Now for a given subgraph $K$ of $K_\w$, the
number of elements $\u\in U_\w(\F_q)$ such that $K=K_{\w,\u}$ equals
$(q-1)^{e(K)}$. Moroever by Proposition \ref{connected}, the group
$T/Z_r$ acts trivially on $\calC$ and so the number of $T$-orbits of
$\calC$ over $\F_q$ equals $R_\w(q)$.
\end{proof}

We can now state the main result of our paper:
\begin{theorem}\label{thm:main}
We have 
$$
A_{\Gamma,\v_\muhat}(q)=\sum_{\w\in S_{\tilde{\muhat}}}R_\w(q).
$$
\end{theorem}

\subsection{Tutte polynomial of graphs.}
\label{rem:tutte}
The above polynomials $R_\w (q)$ are related to classical graph
polynomials.  Recall (cf.~\cite{sokal, godsil}) that the Tutte
polynomial $T_{K} (x,y)\in \Z[x,y]$ for a graph $K$ with edge set $E$
and vertex set $V$ can be defined by
\[
T_{K} (x,y) = \sum_{A\subseteq E} (x-1)^{k (A)-k (E)} (y-1)^{k (A)+|A|-|V|},
\]
where $k (A)$ is the number of connected components of the subgraph
with edge set $A$. Tutte proved that for a connected graph $K$ we also
have
\[
T_K(x,y)=\sum_Tx^{i(T)}y^{e(T)},
\]
where the sum is over all spanning trees $T$ of $K$ and $i(T),e(T)$
are respectively their {\it internal} and {\it external activity} (for
some fixed but arbitrary ordering of the edges of $K$). In particular,
the coefficients of $T(x,y)$ are non-negative integers.

 In this paper we will only be concerned with the specialization (for
 $K$ a connected graph)
\[
R_K(q):=T_K(1,q)=\sum_Tq^{e(T)},
\]
which we will call the {\it external activity polynomial} of $K$.  Up
to a variable change and renormalization, $R_K(q)$ coincides with the
reliability polynomial
\[
(1-p)^{|V|-k(K)}p^{|E|-|V|+k(K)} T_{K}(1,1/p),
\]
which computes the probability that a connected graph $K$ remains
connected when each edge is independently deleted with fixed
probability $p$.

A result of Hausel and Sturmfels~\cite{hausel-sturmfels} implies that
the Kac polynomial of a quiver with dimension vector consisting of all
$1$'s equals the external activity polynomial of the underlying graph.

 It is clear that if $K=K_{\w}$ is connected then
\[
R_{\w} (q) = R_{K_\w}.
\]
Hence Theorem \ref{thm:main} together with Tutte's result provide an
alternative proof of the non-negativity of the coefficients of the Kac
polynomials $A_{\Gamma,\v_\muhat}(q)$ (see Theorem \ref{HLV}).

\subsection{Counting $T$-orbits on flag varieties}

Let $P_1,\dots,P_k$ be parabolic subgroups of $\GL_r$ containing the
lower triangular matrices (this is only for convenience). Recall that
$T$ denotes the maximal torus of $\GL_r$ of diagonal matrices. To each
parabolic $P_i$ corresponds a unique partition
$\tilde{\mu}{^i}=(\tilde{\mu}{^i_1},\tilde{\mu}{^i_2},\dots)$ given by the size of the blocks. Denote
by $E^T_{\tilde{\muhat}}(q)$ the number over $\F_q$ of $T$-orbits in
$\prod_{i=1}^k\GL_r/P_i$ whose stabilizers equal $Z_r$.  For $i=1,\dots,k$,
put $n_i:=r-\tilde{\mu}{^i_1}$, and denote by $\mu^i$ the partition
$(\tilde{\mu}{^i_2},\tilde{\mu}{^i_3},\dots)$ of $n_i$. From the tuple
$\muhat=(\mu^1,\dots,\mu^k)$ and $r$ we consider the associated quiver
$\Gamma$ equipped with dimension vector $\v_\muhat$ as in \S
\ref{dandelion}.

In view of Remark \ref{rem:multiflags}, we deduce from Proposition
\ref{ab} the following result, which relates Kac polynomials of
complete bipartite supernova quivers to counting $T$-orbits:

\begin{theorem} We have 
$$
E^T_{\tilde{\muhat}}(q)=A_{\Gamma,\v_\muhat}(q).
$$
In particular, $E^T_{\tilde{\muhat}}(q)$ is non zero if and only if
$\v_\muhat\in\Phi(\Gamma)$. Moreover $E^T_{\tilde{\muhat}}(q)=1$ if
and only if $\v_\muhat$ is a real root.
\label{maintheo1}\end{theorem}

\begin{remark}\label{rem:speyer}
According to Theorem \ref{maintheo1}, Theorem \ref{thm:main} and \S \ref{rem:tutte} we can count certain $T$-orbits
on homogeneous varieties over $\F_{q}$ in terms of specializations of
Tutte polynomials of certain graphs.  Work of Fink and Speyer
\cite{fs, speyer} provides a geometric interpretation of the Tutte
polynomial of realizable matroids and the $T$-equivariant $K$-theory
of torus orbits.  It would be interesting to understand the
relationship between our work and theirs.
\end{remark}

\section{Examples}\label{sec:examples}
\subsection{Notation}
In this section we present examples to illustrate Theorems
\ref{thm:main} and \ref{maintheo1}.  We first consider the special case when
$k$, the number of long legs of the supernova, equals $1$.  We call
such quivers \emph{dandelion quivers} (cf.~Figure
\ref{fig:dandelion}).  In these examples the tuple of permutations
$\w$ consists of a single element $w$, so we lighten notation and
write $K_{w}$ for $K_{\w}$, etc.  We represent permutations $w \in
S_{r}$ by giving the sequence of their values, using square brackets
to avoid conflict with cycle notation.  Thus $[3,2,4,1] \in S_{4}$
means the permutation taking $1\mapsto 3, 2\mapsto 2, 3\mapsto 4,
4\mapsto 1$.  When possible we omit brackets and commas and write
e.g.~$3241$ for $[3,2,4,1]$.

\newcommand{\sss}{\scriptstyle}
\begin{figure}[htb]
\begin{center}
\psfrag{vd}{$\vdots$}
\psfrag{hd}{$\dots$}
\psfrag{0}{$\sss (1;0)$}
\psfrag{21}{$\sss (1)$}
\psfrag{31}{$\sss (2)$}
\psfrag{41}{$\sss (3)$}
\psfrag{rm11}{$\sss (r-2)$}
\psfrag{r1}{$\sss (r-1)$}
\psfrag{rp11}{$\sss (r)$}
\psfrag{11}{$\sss (1;1)$}
\psfrag{1dm2}{$\sss (1;s_{1}-1)$}
\psfrag{1dm1}{$\sss (1,s_{1})$}
\includegraphics[scale=0.25]{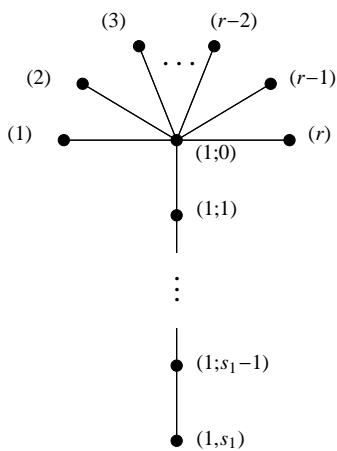}
\end{center}
\caption{The dandelion quiver\label{fig:dandelion}}
\end{figure}

\subsection{Projective space}
Consider the dandelion quiver with no long leg, and with central node
labelled with $n$.  In this example we consider the two cases $r=n$
and $r=n+1$.  It is not hard to see that the corresponding root is
real.  Indeed, apply a reflection at the central node.  If $r=n$ we
get all leaf nodes labelled with $1$ and with the central node
labelled with $0$.  If $r=n+1$, the central node is labelled with $1$.
We can further apply reflections along the leaves to make every leaf
have label $0$.  Thus in these cases the root is real and we should
have $A=1$.

If $r=n$, then the homogeneous variety is that of $n$-planes in
$\kappa^{n}$, i.e.~is a single point.  There is one inversion graph,
which is itself a point, and Theorem \ref{thm:main} implies that 
the Kac polynomial equals 1.

On the other hand, if $r=n+1$, then our homogeneous variety is that of
$n$-planes in $\kappa^{n+1}$, i.e.~is a projective space.  This time
the only connected inversion graph corresponds to the permutation
$w=[n+1,1,2,\dotsc ,n]$, which indexes the open Schubert cell.  The
graph $K_{w}$ is a tree, and again the Kac polynomial equals 1.

\subsection{A grassmannian}\label{ss:grassmannian}
Now we consider a more complicated example.  Let $\Gamma$ be the
quiver in Figure \ref{fig:example1}, with the indicated dimension
vector $\v_{\mu }$.  One can check using Lemma \ref{lem:twofive} that
this vector gives an imaginary root.  The homogeneous variety is $Gr
(2,5)$, the grassmannian of $2$-planes in $\kappa^5$.  This variety is
$6$-dimensional and can be paved by $10$ Schubert cells $U_{w} =
P_{s}wB$, where $w$ ranges over the minimal length elements in the
$10$ cosets of $S_{2}\times S_{3}$ in $S_{5}$.  Hence there are 10
graphs $K_{w}$ of order $5$ that we need for $A_{\Gamma ,\v_\mu} (q)$.
Of these graphs, only $4$ are connected.  In fact, the number of edges
of $K_{w}$ equals the dimension of the Schubert cell $U_{w}$, and
since we must have at least four edges for a graph of order $5$ to be
connected, only the cells of dimensions $\geq 4$ need to be
considered.  These are labelled by the permutations $31452$, $34125$,
$34152$, and $34512$.

Figures \ref{fig:firsttwo}--\ref{fig:secondtwo} show these four
graphs.  We consider each in turn:
\begin{itemize}
\item The graph $K_{34125}$ is not
connected, so $R_{34125}=0$.
\item The graph $K_{31452}$ is a
connected tree, which implies $R_{31452}= 1$.  
\item The graph $K_{34152}$ is a $4$-cycle with an extra edge.  There
are $4$ spanning trees contributing $1$ each, and the full graph
contributes $q-1$.  Thus $R_{34152}=q+3$.
\item The last graph $K_{34512}$ is a complete bipartite graph of type
$(2,3)$.  There are $12$ spanning trees; each contributes $1$ to
$R_{34512}$.  Deleting any single edge yields a graph isomorphic
to $K_{34152}$, each of which contributes $q-1$.  Finally, the full
graph itself has betti number $2$ and thus contributes $(q-1)^{2}$.
Altogether we find $R_{34512}= q^2 + 4q + 7$.
\end{itemize}
Thus
\begin{equation}\label{eq:example1}
A_{\Gamma , \v_\mu} (q) = R_{31452}+ R_{34152}+ R_{34512}= q^{2}+5q+11.
\end{equation}


\begin{figure}[htb]
\begin{center}
\begin{tikzpicture}[scale=1.25]
\coordinate (0) at (0,0);
\coordinate (1) at (0:1);
\coordinate (2) at (45:1);
\coordinate (3) at (90:1);
\coordinate (4) at (135:1);
\coordinate (5) at (180:1);
\draw[thick] (0) -- (1);
\draw[thick] (0) -- (2);
\draw[thick] (0) -- (3);
\draw[thick] (0) -- (4);
\draw[thick] (0) -- (5);
\node at (0,-0.35)  {$2$};
\node at (0:1.35) {$1$};
\node at (45:1.35) {$1$};
\node at (90:1.35) {$1$};
\node at (135:1.35) {$1$};
\node at (180:1.35) {$1$};
\fill[black] (0) circle (0.07);
\fill[black] (1) circle (0.07);
\fill[black] (2) circle (0.07);
\fill[black] (3) circle (0.07);
\fill[black] (4) circle (0.07);
\fill[black] (5) circle (0.07);
\end{tikzpicture}
\end{center}
\caption{\label{fig:example1}}
\end{figure}

\begin{figure}[htb]
\centering
\subfigure[$w=31452$\label{fig:31452}]{
\begin{tikzpicture}[scale=0.75]
\clip (-2.01,-2.01) rectangle (2.01,2.01);
\coordinate (1) at (0:1);
\coordinate (2) at (72:1);
\coordinate (3) at (144:1);
\coordinate (4) at (-144:1);
\coordinate (5) at (-72:1);
\draw[thick] (1) -- (2);
\draw[thick] (1) -- (5);
\draw[thick] (3) -- (5);
\draw[thick] (4) -- (5);
\node at (0:1.35) {$1$};
\node at (72:1.35) {$2$};
\node at (144:1.35) {$3$};
\node at (-144:1.35) {$4$};
\node at (-72:1.35) {$5$};
\fill[black] (1) circle (0.1);
\fill[black] (2) circle (0.1);
\fill[black] (3) circle (0.1);
\fill[black] (4) circle (0.1);
\fill[black] (5) circle (0.1);
\end{tikzpicture}
}
\quad\quad 
\subfigure[$w=34125$\label{fig:34125}]{
\begin{tikzpicture}[scale=0.75]
\clip (-2.01,-2.01) rectangle (2.01,2.01);
\coordinate (1) at (0:1);
\coordinate (2) at (72:1);
\coordinate (3) at (144:1);
\coordinate (4) at (-144:1);
\coordinate (5) at (-72:1);
\draw[thick] (1) -- (3);
\draw[thick] (1) -- (4);
\draw[thick] (2) -- (3);
\draw[thick] (2) -- (4);
\node at (0:1.35) {$1$};
\node at (72:1.35) {$2$};
\node at (144:1.35) {$3$};
\node at (-144:1.35) {$4$};
\node at (-72:1.35) {$5$};
\fill[black] (1) circle (0.1);
\fill[black] (2) circle (0.1);
\fill[black] (3) circle (0.1);
\fill[black] (4) circle (0.1);
\fill[black] (5) circle (0.1);
\end{tikzpicture}
}
\medskip
\caption{\label{fig:firsttwo}}
\end{figure}

\begin{figure}[htb]
\centering
\subfigure[$w=34152$\label{fig:34152}]{
\begin{tikzpicture}[scale=0.75]
\clip (-2.01,-2.01) rectangle (2.01,2.01);
\coordinate (1) at (0:1);
\coordinate (2) at (72:1);
\coordinate (3) at (144:1);
\coordinate (4) at (-144:1);
\coordinate (5) at (-72:1);
\draw[thick] (1) -- (3);
\draw[thick] (1) -- (5);
\draw[thick] (2) -- (3);
\draw[thick] (2) -- (5);
\draw[thick] (4) -- (5);
\node at (0:1.35) {$1$};
\node at (72:1.35) {$2$};
\node at (144:1.35) {$3$};
\node at (-144:1.35) {$4$};
\node at (-72:1.35) {$5$};
\fill[black] (1) circle (0.1);
\fill[black] (2) circle (0.1);
\fill[black] (3) circle (0.1);
\fill[black] (4) circle (0.1);
\fill[black] (5) circle (0.1);
\end{tikzpicture}
}
\quad\quad 
\subfigure[$w=34512$\label{fig:34512}]{
\begin{tikzpicture}[scale=0.75]
\clip (-2.01,-2.01) rectangle (2.01,2.01);
\coordinate (1) at (0:1);
\coordinate (2) at (72:1);
\coordinate (3) at (144:1);
\coordinate (4) at (-144:1);
\coordinate (5) at (-72:1);
\draw[thick] (1) -- (4);
\draw[thick] (1) -- (5);
\draw[thick] (2) -- (4);
\draw[thick] (2) -- (5);
\draw[thick] (3) -- (4);
\draw[thick] (3) -- (5);
\node at (0:1.35) {$1$};
\node at (72:1.35) {$2$};
\node at (144:1.35) {$3$};
\node at (-144:1.35) {$4$};
\node at (-72:1.35) {$5$};
\fill[black] (1) circle (0.1);
\fill[black] (2) circle (0.1);
\fill[black] (3) circle (0.1);
\fill[black] (4) circle (0.1);
\fill[black] (5) circle (0.1);
\end{tikzpicture}
}
\medskip
\caption{\label{fig:secondtwo}}
\end{figure}

\subsection{A two-step flag variety}
\label{two-step}
Now consider the dandelion quiver in Figure \ref{fig:example2}, with
the indicated dimension vector.  This is of course the same example we
just treated, except that now we regard one of the short legs as being
the long leg.  The corresponding homogeneous variety is no longer a a
grassmannian; instead we have the partial flag variety of two-step
flags $E^{3}\subset E^{2}$ in $\kappa^{4}$.  This time the inversion
graphs have $4$ vertices, so we need at least $3$ edges in any $K_{w}$
for it be connected, and there are 6 permutations with at least three
inversions.  The graphs are show in
Figures~\ref{fig:g124one}--\ref{fig:g124three}.  We leave it to the
reader to check the following:
\begin{itemize}
\item $R_{3142}=1$
\item $R_{3214}=0$
\item $R_{3412}=q+3$
\item $R_{2341}=1$
\item $R_{3241}=q+2$
\item $R_{3421}=q^{2}+3q+4$
\end{itemize}
Thus
\begin{equation}\label{eq:example2}
A_{\Gamma , \v_\mu} (q) = q^{2}+5q+11,
\end{equation}
which agrees with \eqref{eq:example1}. 


\begin{figure}[htb]
\begin{center}
\begin{tikzpicture}[scale=1.25]
\coordinate (0) at (0,0);
\coordinate (1) at (0:1);
\coordinate (2) at (60:1);
\coordinate (3) at (120:1);
\coordinate (4) at (180:1);
\coordinate (5) at (-90:1);
\draw[thick] (0) -- (1);
\draw[thick] (0) -- (2);
\draw[thick] (0) -- (3);
\draw[thick] (0) -- (4);
\draw[thick] (0) -- (5);
\node at (-0.35,-0.35)  {$2$};
\node at (0:1.35) {$1$};
\node at (60:1.35) {$1$};
\node at (120:1.35) {$1$};
\node at (180:1.35) {$1$};
\node at (-90:1.35) {$1$};

\fill[black] (0) circle (0.07);
\fill[black] (1) circle (0.07);
\fill[black] (2) circle (0.07);
\fill[black] (3) circle (0.07);
\fill[black] (4) circle (0.07);
\fill[black] (5) circle (0.07);
\end{tikzpicture}
\end{center}
\caption{\label{fig:example2}}
\end{figure}

\begin{figure}[htb]
\centering
\subfigure[$w=3142$\label{fig:3142}]{
\begin{tikzpicture}[scale=0.75]
\clip (-2.01,-2.01) rectangle (2.01,2.01);
\coordinate (1) at (0:1);
\coordinate (2) at (90:1);
\coordinate (3) at (180:1);
\coordinate (4) at (270:1);
\draw[thick] (1) -- (2);
\draw[thick] (1) -- (4);
\draw[thick] (3) -- (4);
\node at (0:1.35) {$1$};
\node at (90:1.35) {$2$};
\node at (180:1.35) {$3$};
\node at (270:1.35) {$4$};
\fill[black] (1) circle (0.1);
\fill[black] (2) circle (0.1);
\fill[black] (3) circle (0.1);
\fill[black] (4) circle (0.1);
\end{tikzpicture}
}
\quad\quad 
\subfigure[$w=3214$\label{fig:3214}]{
\begin{tikzpicture}[scale=0.75]
\clip (-2.01,-2.01) rectangle (2.01,2.01);
\coordinate (1) at (0:1);
\coordinate (2) at (90:1);
\coordinate (3) at (180:1);
\coordinate (4) at (270:1);
\draw[thick] (1) -- (2);
\draw[thick] (1) -- (3);
\draw[thick] (2) -- (3);
\node at (0:1.35) {$1$};
\node at (90:1.35) {$2$};
\node at (180:1.35) {$3$};
\node at (270:1.35) {$4$};
\fill[black] (1) circle (0.1);
\fill[black] (2) circle (0.1);
\fill[black] (3) circle (0.1);
\fill[black] (4) circle (0.1);
\end{tikzpicture}
}
\medskip
\caption{\label{fig:g124one}}
\end{figure}

\begin{figure}[htb]
\centering
\subfigure[$w=3412$\label{fig:3412}]{
\begin{tikzpicture}[scale=0.75]
\clip (-2.01,-2.01) rectangle (2.01,2.01);
\coordinate (1) at (0:1);
\coordinate (2) at (90:1);
\coordinate (3) at (180:1);
\coordinate (4) at (270:1);
\draw[thick] (1) -- (3);
\draw[thick] (1) -- (4);
\draw[thick] (2) -- (3);
\draw[thick] (2) -- (4);
\node at (0:1.35) {$1$};
\node at (90:1.35) {$2$};
\node at (180:1.35) {$3$};
\node at (270:1.35) {$4$};
\fill[black] (1) circle (0.1);
\fill[black] (2) circle (0.1);
\fill[black] (3) circle (0.1);
\fill[black] (4) circle (0.1);
\end{tikzpicture}
}
\quad\quad 
\subfigure[$w=2341$\label{fig:2341}]{
\begin{tikzpicture}[scale=0.75]
\clip (-2.01,-2.01) rectangle (2.01,2.01);
\coordinate (1) at (0:1);
\coordinate (2) at (90:1);
\coordinate (3) at (180:1);
\coordinate (4) at (270:1);
\draw[thick] (1) -- (4);
\draw[thick] (3) -- (4);
\draw[thick] (2) -- (4);
\node at (0:1.35) {$1$};
\node at (90:1.35) {$2$};
\node at (180:1.35) {$3$};
\node at (270:1.35) {$4$};
\fill[black] (1) circle (0.1);
\fill[black] (2) circle (0.1);
\fill[black] (3) circle (0.1);
\fill[black] (4) circle (0.1);
\end{tikzpicture}
}
\medskip
\caption{\label{fig:g124two}}
\end{figure}

\begin{figure}[htb]
\centering
\subfigure[$w=3241$\label{fig:3241}]{
\begin{tikzpicture}[scale=0.75]
\clip (-2.01,-2.01) rectangle (2.01,2.01);
\coordinate (1) at (0:1);
\coordinate (2) at (90:1);
\coordinate (3) at (180:1);
\coordinate (4) at (270:1);
\draw[thick] (1) -- (2);
\draw[thick] (1) -- (4);
\draw[thick] (2) -- (4);
\draw[thick] (3) -- (4);
\node at (0:1.35) {$1$};
\node at (90:1.35) {$2$};
\node at (180:1.35) {$3$};
\node at (270:1.35) {$4$};
\fill[black] (1) circle (0.1);
\fill[black] (2) circle (0.1);
\fill[black] (3) circle (0.1);
\fill[black] (4) circle (0.1);
\end{tikzpicture}
}
\quad\quad 
\subfigure[$w=3421$\label{fig:3421}]{
\begin{tikzpicture}[scale=0.75]
\clip (-2.01,-2.01) rectangle (2.01,2.01);
\coordinate (1) at (0:1);
\coordinate (2) at (90:1);
\coordinate (3) at (180:1);
\coordinate (4) at (270:1);
\draw[thick] (1) -- (3);
\draw[thick] (1) -- (4);
\draw[thick] (2) -- (3);
\draw[thick] (2) -- (4);
\draw[thick] (3) -- (4);
\node at (0:1.35) {$1$};
\node at (90:1.35) {$2$};
\node at (180:1.35) {$3$};
\node at (270:1.35) {$4$};
\fill[black] (1) circle (0.1);
\fill[black] (2) circle (0.1);
\fill[black] (3) circle (0.1);
\fill[black] (4) circle (0.1);
\end{tikzpicture}
}
\medskip
\caption{\label{fig:g124three}}
\end{figure}

\subsection{A product of projective planes}\label{ss:p2squared}
Now we consider a more general supernova quiver.  We take
$r=3$ and $(n_{1},n_{2})= (1,1)$.  Thus the quiver is the complete
bipartite graph of type $(3,2)$, and the dimension vector assigns $1$
to each vertex.  In terms of $T$-orbits, we are counting the orbits of
dimension $2$ on a product of two projective planes with a
$2$-dimensional torus acting diagonally.

The inversion graphs are labelled by pairs of permutations
$(w_{1},w_{2})\in (S_{3})^{2}$.  There are five connected inversion
graphs; they are characterized by having at least one $w_{i}$ equal to
$312$, the longest permutation for this Bruhat decomposition.  We show
the graphs in Figures \ref{fig:firsttwomulti}--\ref{fig:312312} (edges
curving in correspond to the first permutation, and those curving out
to the second). We find
\begin{itemize}
\item $R_{123,312} = R_{312,123} = 1$
\item $R_{132,312} = R_{312,132} = q+2$
\item $R_{312,312} = q^{2}+2q+1$
\end{itemize}
Altogether we obtain
\begin{equation}\label{eq:pair}
A_{\Gamma , \v_{\muhat}} = q^{2}+4q+7.
\end{equation}
We remark that \eqref{eq:pair} is in fact the external activity
polynomial of the underlying graph of the quiver thanks to the result
of Hausel and Sturmfels (see~\S\ref{rem:tutte}).  Indeed, the Tutte
polynomial of the complete bipartite graph of type $(3,2)$ is
\[
x^4 + 2x^3 + 3x^2 + x + y^{2} + 4y.
\]

We can also recover \eqref{eq:pair} by counting $2$-dimensional torus
orbits in $\calF = \bP^{2}\times \bP^{2}$, following Theorem
\ref{maintheo1}.  Let $\pi \colon \calF \rightarrow \bP^{2}$ be the
projection onto the first factor.  The action of the torus $T$
commutes with $\pi$.
\begin{itemize}
    \item Choose a point $p_{0}$ in the image of $\pi$ with trivial
    stabilizer. Any point in the inverse image of $p_{0}$ determines a
    unique $2$-dimensional orbit, and thus this accounts for
    $q^{2}+q+1$ orbits.
    \item Now choose a point $p_{0}$ in the image of $\pi$ with
    $1$-dimensional stabilizer.  We claim the inverse image of $p_{0}$
    determines $q+1$ orbits.  Indeed, after we have fixed $p_{0}$,
    have one dimension of $T$ left.  This can move points along the
    lines in $T$-fixed point not contained in the closure of the orbit
    of $p_{0}$.  There are $q+1$ such lines, and hence $q+1$ orbits.
    Since there are $3$ choices for $p_{0}$ (corresponding to the
    three $1$-dimensional $T$ orbits in $\bP^{2}$ we obtain $3q+3$
    orbits altogether.
    \item Finally we can choose a point $p_{0}$ fixed by $T$.  There
    is one $2$-dimensional $T$-orbit in the inverse image of $p_{0}$.
    Since there are $3$ choices of $p_{0}$ we get $3$ orbits this way.
\end{itemize}
Hence altogether we find $q^{2}+4q+7$ torus orbits of dimension $2$,
which coincides with \eqref{eq:pair}.

\begin{figure}[htb]
\centering
\subfigure[$(123,312)$\label{fig:123312}]{
\begin{tikzpicture}[scale=0.75]
\clip (-2.01,-2.01) rectangle (2.01,2.01);
\coordinate (1) at (0:1);
\coordinate (2) at (120:1);
\coordinate (3) at (240:1);
\coordinate (c) at (60:0.9);
\coordinate (a) at (180:0.9);
\coordinate (b) at (300:0.9);
\coordinate (C) at (60:0.25);
\coordinate (A) at (180:0.25);
\coordinate (B) at (300:0.25);

\draw[thick] (1) .. controls (c) .. (2);
\draw[thick] (1) .. controls (b) .. (3);
\node at (0:1.35) {$1$};
\node at (120:1.35) {$2$};
\node at (240:1.35) {$3$};
\fill[black] (1) circle (0.1);
\fill[black] (2) circle (0.1);
\fill[black] (3) circle (0.1);
\end{tikzpicture}
}
\quad\quad 
\subfigure[$(312,123)$\label{fig:312123}]{
\begin{tikzpicture}[scale=0.75]
\clip (-2.01,-2.01) rectangle (2.01,2.01);
\coordinate (1) at (0:1);
\coordinate (2) at (120:1);
\coordinate (3) at (240:1);
\coordinate (c) at (60:0.9);
\coordinate (a) at (180:0.9);
\coordinate (b) at (300:0.9);
\coordinate (C) at (60:0.25);
\coordinate (A) at (180:0.25);
\coordinate (B) at (300:0.25);

\draw[thick] (1) .. controls (C) .. (2);
\draw[thick] (1) .. controls (B) .. (3);
\node at (0:1.35) {$1$};
\node at (120:1.35) {$2$};
\node at (240:1.35) {$3$};
\fill[black] (1) circle (0.1);
\fill[black] (2) circle (0.1);
\fill[black] (3) circle (0.1);
\end{tikzpicture}
}
\medskip
\caption{\label{fig:firsttwomulti}}
\end{figure}

\begin{figure}[htb]
\centering
\subfigure[$(132,312)$\label{fig:132312}]{
\begin{tikzpicture}[scale=0.75]
\clip (-2.01,-2.01) rectangle (2.01,2.01);
\coordinate (1) at (0:1);
\coordinate (2) at (120:1);
\coordinate (3) at (240:1);
\coordinate (c) at (60:0.9);
\coordinate (a) at (180:0.9);
\coordinate (b) at (300:0.9);
\coordinate (C) at (60:0.25);
\coordinate (A) at (180:0.25);
\coordinate (B) at (300:0.25);

\draw[thick] (1) .. controls (c) .. (2);
\draw[thick] (1) .. controls (b) .. (3);
\draw[thick] (2) .. controls (A) .. (3);
\node at (0:1.35) {$1$};
\node at (120:1.35) {$2$};
\node at (240:1.35) {$3$};
\fill[black] (1) circle (0.1);
\fill[black] (2) circle (0.1);
\fill[black] (3) circle (0.1);
\end{tikzpicture}
}
\quad\quad 
\subfigure[$(312,132)$\label{fig:312132}]{
\begin{tikzpicture}[scale=0.75]
\clip (-2.01,-2.01) rectangle (2.01,2.01);
\coordinate (1) at (0:1);
\coordinate (2) at (120:1);
\coordinate (3) at (240:1);
\coordinate (c) at (60:0.9);
\coordinate (a) at (180:0.9);
\coordinate (b) at (300:0.9);
\coordinate (C) at (60:0.25);
\coordinate (A) at (180:0.25);
\coordinate (B) at (300:0.25);

\draw[thick] (1) .. controls (C) .. (2);
\draw[thick] (1) .. controls (B) .. (3);
\draw[thick] (2) .. controls (a) .. (3);
\node at (0:1.35) {$1$};
\node at (120:1.35) {$2$};
\node at (240:1.35) {$3$};
\fill[black] (1) circle (0.1);
\fill[black] (2) circle (0.1);
\fill[black] (3) circle (0.1);
\end{tikzpicture}
}
\medskip
\caption{\label{fig:firsttwo}}
\end{figure}

\begin{figure}[h]
\begin{center}
\begin{tikzpicture}[scale=0.75]
\clip (-2.01,-2.01) rectangle (2.01,2.01);
\coordinate (1) at (0:1);
\coordinate (2) at (120:1);
\coordinate (3) at (240:1);
\coordinate (c) at (60:0.9);
\coordinate (a) at (180:0.9);
\coordinate (b) at (300:0.9);
\coordinate (C) at (60:0.25);
\coordinate (A) at (180:0.25);
\coordinate (B) at (300:0.25);

\draw[thick] (1) .. controls (C) .. (2);
\draw[thick] (1) .. controls (B) .. (3);
\draw[thick] (1) .. controls (c) .. (2);
\draw[thick] (1) .. controls (b) .. (3);

\node at (0:1.35) {$1$};
\node at (120:1.35) {$2$};
\node at (240:1.35) {$3$};
\fill[black] (1) circle (0.1);
\fill[black] (2) circle (0.1);
\fill[black] (3) circle (0.1);
\end{tikzpicture}
\end{center}
\caption{$(312,312)$\label{fig:312312}}
\end{figure}

\subsection{Counting $T$-orbits}\label{ss:torbits}

We conclude by illustrating Theorem \ref{maintheo1} for the grassmannian $Gr
(2,5)$ from section \ref{ss:grassmannian}.  The main tool we use is
the \emph{Gel$'$fand--MacPherson correspondence}, which we state in
Theorem~\ref{thm:gm}.  We refer to \cite[]{kapranov, gm,ggms,gs} for
more details.

Let $E\subset \C^{r}$ be a subspace of dimension $k$.  Assume that $E$
does not lie in any of the coordinate hyperplanes $H_{i} = \{z_{i}=0
\}\subset \C^{r}$.  The intersections $E\cap H_{i}$ determine a
collection of $r$ hyperplanes in $E$ and thus a point in
$(\bP^{k-1})^{r}$, i.e.~a projective configuration.  (Here we think
of $\bP^{k-1}$ as being $\bP (E^{*})$).  If $E'$ is a
$T$-translate of $E$, then the configuration corresponding to $E'$ is
equivalent to $E$ an element of $\PGL_k$ acting diagonally on
$(\bP^{k-1})^{r}$.   

Hence we can study $T$-orbits on $G (k,r)$ in
terms of certain configurations of $r$ points in $\bP^{k-1}$.  The
precise statement of this fact is the Gel$'$fand--MacPherson
correspondence.  We will only need to understand what happens when the 
the $T$-orbits have maximal dimension $r-1$.

\begin{theorem}\label{thm:gm}
Let $G_{\circ} (k,r)\subset G (k,r)$ be the subset of all $L$ such
that $T\cdot L$ has dimension $r-1$.  Let $(\bP^{k-1})^{r}_{\circ}$ be
the subset of configurations $p= (p_{1},\dotsc ,p_{r})$ such that
$\PGL_{k}\cdot p$ has dimension $k^{2}-1$.  Then the assigment
$L\mapsto p$, where $p_{i}=E\cap H_{i}$, defines a bijection of orbit
spaces
\[
\Phi \colon G_{\circ} (k,r)/T \longrightarrow (\bP^{k-1})^{r}_{\circ}/\PGL_{k}.
\]
\end{theorem}

\begin{remark}
The bijection $\Phi$ can be extended to all of $G (k,r)$
\cite[Proposition 1.5]{ggms}.
\end{remark}

In general it is very difficult to determine the configurations in the
image of $\Phi$, but there is one case that is easy: the grassmannians
$G (2,r)$.  When $k=2$ the configurations are sets of points in the
projective line, and the only degenerations that can occur are
multiple points.  To make this precise, let us say that a collection
of distinct points $p_{1},\dotsc ,p_{m}$ is $r$-labelled if it is
equipped with a surjective map $\{1,\dotsc ,r \}\rightarrow
\{p_{1},\dotsc ,p_{m} \}$.  We have the following characterization of
the $T$-orbits (cf.~\cite[Section~1.3]{kapranov}).

\begin{proposition}
Torus orbits in $G (2,r)$ of maximal dimension are in bijection with
$r$-labelled sets of $m$ points in $\bP^{1}$ up to
$\PGL_{2}$-equivalence, where $3\leq m\leq r$.
\end{proposition}

Now we consider configurations over $\F_{q}$.  Let $C_{m} (q)$ be the
number of configurations of $m$ distinct points up to equivalence.
Fix three points in $\bP^{1} (\F _{q})$ and call them $0$, $1$, and
$\infty$.  Given $m$ unlabelled points in $\bP^{1}$, we can use
$\PGL_2$ to carry three of them to $0,1,\infty$.  This uses up all the
automorphisms, which gives the following:
\[
C_{m} (q) = \begin{cases}
(q-2) (q-3) (q- (m-2))&\text{if $m>3$,}\\
1&\text{if $m=3$.}
\end{cases}
\]
To complete the count we need to incorporate the labellings.  An
$r$-labelling is determined by a sujective map $\{1,\dotsc ,r
\}\rightarrow \{p_{1},\dotsc ,p_{m} \}$, in other words an equivalence
relation on $\{1,\dotsc ,r \}$ with $m$ classes.  These are counted by
$S (r,m)$, the Stirling number of the second kind.  Letting $E^{T}_{r}
(q)$ denote the number of $T$-orbits, we have
\[
E^{T}_{r} (q) = \sum_{m=3}^r S(r,m) C_{m} (q). 
\]
For instance, when $r=5$, we have 
\[
E^{T}_{5} (q) = 1\cdot (q-2) (q-3)+10\cdot (q-2)+25 = q^{2}+5q+11,
\]
in agreement with \eqref{eq:example1}.

Comparing Figures \ref{fig:example1} and \ref{fig:example2}, one sees
that over $\F_{q}$ the number of $(r-1)$-dimensional torus orbits in
$Gr (2,r)$ equals the number of $(r-2)$-dimensional torus orbits in
the flag variety of $\{\text{point} \subset \text{line}\}$ in
$\bP^{r-2}$ (the tori have different dimensions, of course).  This
suggests that there should be a bijection between the sets of torus
orbits for these two homogeneous varieties.  This is true, and we
leave the reader the pleasure of finding it.


\section{Generating Functions}

We will use the series~\cite[(1.4)]{hausel-letellier-villegas} to obtain a generating
function for the Kac polynomials of the supernova quivers
of~\S\ref{dandelion}.  The series~\cite[(1.4)]{hausel-letellier-villegas} in the case where
the quiver is the complete $(k,r)$ bipartite graph with $k+r$ vertices
is the following
\begin{equation}
 \label{maingen}
 \bH(\bfX,\bfY;q):=(q-1)\Log\left(
   \sum_{\lambda^i,\mu^j} 
   \frac{q^{\sum_{i,j}{\langle\lambda^i,\mu^j\rangle}}
  \prod_i  \tilde{H}_{\lambda^i}(\x_i;q) 
   \prod_j  \tilde{H}_{\mu^j}(\y_j;q) }
{\prod_i
     q^{\langle\lambda^i,\lambda^i\rangle}b_{\lambda^i}(q^{-1})
\prod_j 
     q^{\langle\mu^j,\mu^j\rangle}b_{\mu^j}(q^{-1})}
 \right), 
\end{equation}
where $i=1,\ldots,k,j=1,\ldots,r$
$$
b_\lambda(q):=    \prod_{i\geq 1}\prod_{j=1}^{m_i(\lambda)}(1-q^j), 
$$
with $m_i(\lambda)$ the multiplicity of $i$ in $\lambda$ and 
$\bfX= (\x_1,\dots,\x_r);\bfY=(\y_1,\ldots,\y_k)$.

Since we are interested in a dimension vector where the $r$ vertices
have value $1$ we can restrict the $\y$ variables to
$\y_i=(u_i,0,\dots)$ for some independent variables $u_1,\ldots,
u_r$. Furthermore, we only need to work modulo the ideal $I:=\langle
u_1^2,\ldots,u_r^2\rangle$.

We have $\tilde{H}_\lambda(u,0,\ldots)=u^{|\lambda|}$. It follows that
the right hand side of~\eqref{maingen} becomes
$$
(q-1)\Log\left(
   \sum_{\lambda^i} 
\sum_{s=0}^r
   \frac{q^{s\sum_il(\lambda^i)}
e_s(u)
  \prod_i  \tilde{H}_{\lambda^i}(\x_i;q)}
{(q-1)^s\prod_i
     q^{\langle\lambda^i,\lambda^i\rangle}b_{\lambda^i}(q^{-1})
 }
 \right) \bmod I, 
$$
where $e_s(u)=e_s(u_1,\ldots,u_r)$ is the elementary symmetric function in
the $u_i$'s. Interchanging summations this equals
$$
(q-1)\Log\left(
\sum_{s=0}^r
\prod_i
c_s(\x_i)
\frac{e_s(u)}
{(q-1)^s}
 \right) \bmod I, 
$$
where
$$
c_s(\x):=   \sum_\lambda\frac{q^{sl(\lambda)}
   \tilde{H}_{\lambda}(\x;q)}
{q^{\langle\lambda,\lambda\rangle}b_{\lambda}(q^{-1})}, \qquad 
\x=(x_1,x_2,\ldots).
$$
Note that
$$
e_{s_1}(u)\cdots e_{s_l}(u)\equiv 
\frac{(s_1+\cdots s_l)!}{s_1!\cdots s_l!}e_{s_1+\cdots
  +s_l}(u) \bmod I\,.
$$
Therefore we may replace $e_s(u)$ by a single term $U^s/s!$ and  let
$r$ be arbitrary. Except for the constant term in $U$ the
values of $\Log$ and $\log$ agree since we are working modulo
$I$. Hence we get
$$
(q-1)\Log\left(\prod_i
c_0(\x_i)\right)+ (q-1)\log\left(
\sum_{s\geq 0}
\prod_i
\frac{c_s(\x_i)}{c_0(\x_i)}
\frac{(U/(q-1))^s}
{s!}
 \right) 
$$
Define the \emph{Rogers-Sz\"ego symmetric functions} as
$$
R_s(\x):=\sum_{|\lambda|=s}
{s\brack \lambda_1,\lambda_2,\cdots}
\ m_\lambda(\x),
\qquad  
\x:=(x_1,x_2,\ldots),
$$
where $m_\lambda$ is the monomial symmetric function and
$$
{s\brack \lambda_1,\lambda_2,\cdots}:=
\frac{[s]!}{[\lambda_1]![\lambda_2]!\cdots}, 
\qquad [n]!:=(1-q)\cdots(1-q^n),
$$
is the $q$-multinomial and $q$-factorial respectively.
\begin{proposition}
The following identity holds
$$
\frac{c_s(\x;q)}{c_0(\x;q)}=R_s(1,x_1,x_2,\ldots), \qquad 
\x:=(x_1,x_2,\ldots).
$$
\end{proposition}
Let $\calA_s(\x_1,\ldots,\x_k;q)$ be defined by the generating
function
\begin{equation}
\label{A-defn}
\sum_{s\geq 1} \calA_s(\x_1,\ldots,\x_k;q) \;\frac{U^s}{s!}  =(q-1)\log
\sum_s R_s(\x_1)\cdots R_s(\x_k) \frac{(U/(q-1))^s}{s!}.
\end{equation}
\begin{proof}
It follows from the main formula proved in~\cite{hausel}.
\end{proof}
A priori $\calA_s(\x_1,\ldots,\x_k;q)$ are symmetric functions with
coefficients in $\Q(q)$. In fact, the coefficients are in $\Z[q]$ as
we now see. Combining the above discussion
with~\cite{hausel-letellier-villegas}[Prop. (1.3) (i)] we finally
obtain the following.
\begin{theorem}
  With the notation of~\S\ref{Kac-polyns} the Kac polynomial
  $A_{\Gamma,\v_\muhat}$ of the complete bipartite supernova quiver is
  given by
\begin{equation}
\label{gf-fmla}
A_{\Gamma,\v_\muhat}(q)=\langle \calA_r,h_{\tilde\muhat}\rangle,
\end{equation}
where $h_\mu$ denotes the complete symmetric function,
$h_{\tilde\muhat}:=h_{\tilde\mu^1}\cdots h_{\tilde\mu^k}$ with 
$\tilde\muhat=(\tilde\mu^1,\ldots,\tilde\mu^k)$ and $\tilde\mu^i$ is
the partition of $r$ defined by $(r-|\mu^i|,\mu_1^i,\mu_2^i,\ldots)$.
\end{theorem}
The right hand side of~\eqref{gf-fmla} gives the coefficient of
$m_{\tilde\muhat}$ when writing $\calA_r$ in terms of the monomial
symmetric functions. For example, for $k=1$ we obtain the following 

\bigskip
\begin{tabular}{CCL}
\calA_1&=& m_{1}\\
\calA_2 &=& m_{1^2}\\
\calA_3&=&(q+4)m_{1^3}+m_{12}\\
\calA_4&=&(q^3 + 6q^2 + 20q + 33)m_{1^4}+(q^2 + 5q + 11)m_{1^22}
+(q+4)m_{2^2}+m_{13}
\end{tabular}

\bigskip
\noindent
In particular we see the polynomial $q^2+5q+11$ corresponding to the
example discussed in~\S\ref{two-step}. The coefficient of $m_{1^4}$ on
the other hand corresponds to a dandelion quiver with four short legs
and a long leg with dimension vector $(3,2,1)$ along its vertices
corresponding to the full flag variety $\GL_4/B$. Here is the list of
permutations $w$ of block structure $(1,1,1,1)$ with connected
inversion graphs and their corresponding $R$-polynomials.

\bigskip
\begin{center}
\begin{tabular}{C|R}
w & R_w \\
\hline
  4321 &   q^3 + 3q^2 + 6q + 6\\
  4312 &   q^2 + 3q + 4 \\
  4231&   q^2 + 3q + 4 \\
  4213 &   q + 2 \\
  4132 &   q + 2 \\
  4123 &   1 \\
  3421 &   q^2 + 3q + 4 \\
  3412 &   q + 3 \\
  3241 &   q + 2 \\
  3142 &   1 \\
  2431 &   q + 2 \\
  2413 &   1 \\
 2341 &   1\\
\end{tabular}
\end{center}

\bigskip
We verify that indeed the sum of these polynomials is $ q^3 + 6q^2 +
20q + 33$.

\bibliographystyle{amsplain} \bibliography{supernova}
\end{document}